\documentclass[a4paper]{amsart}
\usepackage[latin1]{inputenc}
\usepackage{amssymb}
\usepackage{amsmath}
\usepackage{mathrsfs}
\usepackage{eufrak}
\usepackage{amsthm}
\usepackage{amsfonts}
\usepackage{textcomp}
\usepackage{graphicx}
\usepackage[pdftex]{color}
\usepackage{paralist}
\usepackage[shortlabels]{enumitem}
\usepackage[hyperpageref]{backref}
\usepackage[pagebackref]{hyperref}
\renewcommand*{\backref}[1]{}  
   \renewcommand*{\backrefalt}[4]{
      \ifcase #1 
         Not cited.
      \or
         Cited on page #2.
      \else
         Cited on pages #2.
      \fi}
\usepackage{comment}
\usepackage[arrow, matrix, curve]{xy}

\usepackage{adjustbox}

\usepackage{tikz} \usetikzlibrary{matrix,patterns,shapes,decorations.pathmorphing,decorations.pathreplacing,calc,arrows} 
\usepackage{tikz-cd} 
\usepackage{tkz-graph}

\newtheorem*{corollary*}{Corollary}
\newtheorem{theorem}{Theorem}[section]

\newtheorem{lemma}[theorem]{Lemma}
\newtheorem{proposition}[theorem]{Proposition}

\newtheorem*{claim*}{Claim}

\theoremstyle{definition}
\newtheorem{definition}[theorem]{Definition}

\newtheorem{main conjecture}[theorem]{Main Conjecture}
\newtheorem*{theorem }{Theorem}

\theoremstyle{remark}

\numberwithin{equation}{section}

\makeatletter
\renewcommand*\env@matrix[1][\
arraystretch]{%
  \edef\arraystretch{#1}%
  \hskip -\arraycolsep
  \let\@ifnextchar\new@ifnextchar
  \array{*\c@MaxMatrixCols c}}
\makeatother

\renewcommand{\mod}{\operatorname{mod}}

\newcommand{\Ext}{\operatorname{Ext}}

\newcommand{\End}{\operatorname{End}}

\newcommand{\Hom}{\operatorname{Hom}}
\newcommand{\add}{\operatorname{\mathrm{add}}}

\newcommand{\op}{\operatorname{\mathrm{op}}}

\newcommand{\id}{\operatorname{\mathrm{id}}}

\renewcommand{\mod}{\operatorname{mod}}

\newcommand{\Ab}{{\mathcal{A}b}}
\newcommand{\Z}{\mathbb{Z}}
\newcommand{\C}{\mathcal{C}}

\newcommand{\CC}{\mathcal{C}}

\newcommand{\Ker}{\mathrm{Ker}}
\newcommand{\Db}{\mathrm{D}^{\mathrm{b}}}

\newcommand{\domdim}{\operatorname{domdim}}

\newcommand{\gldim}{\operatorname{gldim}}

\begin{document}

\title[{Total preprojective algebras}]{Total preprojective algebras}
\date{\today}

\subjclass[2010]{16G10}

\keywords{preprojective algebra, cluster tilting module, tensor algebra, Auslander algebra}

\author[Chan]{Aaron Chan}
\address[Chan]{Graduate School of Mathematics, Nagoya University, Furocho, Chikusaku, Nagoya 464-8602, Japan}
\email{aaron.kychan@gmail.com}

\author[Iyama]{Osamu Iyama}
\address[Iyama]{Graduate School of Mathematical Sciences, University of Tokyo, 3-8-1 Komaba Meguro-ku Tokyo 153-8914, Japan}
\email{iyama@ms.u-tokyo.ac.jp}

\author[Marczinzik]{Ren\'{e} Marczinzik}
\address[Marczinzik]{Mathematical Institute of the University of Bonn, Endenicher Allee 60, 53115 Bonn, Germany}
\email{marczire@math.uni-bonn.de}

\dedicatory{Dedicated to the 80th birthday of Claus Michael Ringel}

\begin{abstract}
We introduce total preprojective algebras $\Psi$ of path algebras of Dynkin quivers $kQ$, and prove that they are isomorphic to $2$-Auslander algebras of preprojective algebras $\Pi$ of $kQ$. In particular, $\Psi$ has global dimension $3$ and dominant dimension $3$. We also describe $\Psi$ as a tensor algebra of a certain explicit bimodule over the Auslander algebra of $kQ$. As an application, we give a presentation of $\Psi$ by explicit quivers with relations. More generally, we introduce total $(d+1)$-preprojective algebras of $d$-representation finite algebras, and give all the corresponding results.
\end{abstract}

\maketitle
 
\section{Our results}


Preprojective algebras are now classical objects of study in representation theory with many applications and connections to other areas, such as Cohen--Macaulay modules \cite{Aus86, GL91}\cite[Chapter 5]{Sko}, Kleinian singularities \cite{Cra00,CBH}, cluster algebras \cite{GLS13}, quantum groups \cite{KS97, Lus91}, and quiver varieties \cite{Nak94}.
Recently they were used in \cite{CIM} to give the first non-trivial classes of Cohen-Macaulay algebras in the sense of Auslander-Reiten \cite{AR}. 
Also, other variations of preprojective algebras including Calabi-Yau completions \cite{K} and generalized preprojective algebras \cite{GLS3} have been actively studied.

In this paper, all modules are right modules and the composition of morphisms $f:X\to Y$ and $g:Y\to Z$ is denoted by $g\circ f:X\to Z$. Thus $X$ is an $\End_A(X)^{\op}$-module. The composition of arrows $a:i\to j$ and $b:j\to k$ in a quiver is denoted by $ba$.

We start with the combinatorial definition of preprojective algebras. For a graph $\Delta$, we fix an orientation to obtain a quiver $Q$. Then define the double $\overline{Q}$  by adding a new arrow $a^*:j\to i$ for each arrow $a:i\to j$ in $Q$. The \emph{preprojective algebra} of $\Delta$ is defined by
\[\Pi=\Pi(\Delta):=k\overline{Q}/\langle\sum_{a\in Q_1}(a^*a-aa^*)\rangle.\]
Then $\Pi$ does not depend on a choice of the orientation of $\Delta$.
There are two alternative definitions of $\Pi$ that are more homological and important for the aim of this article. We give the second one in a slightly more general form.
Let $H$ be a hereditary algebra, e.g. the path algebra $H=kQ$ of an acyclic quiver $Q$. The \emph{inverse Auslander-Reiten translation} is the functor
\[\tau^-:=\Ext^1_H(DH,-):\mod H\to\mod H.\]
Now we introduce the following class of algebras.

\begin{definition}
For $X\in\mod H$, we define the \emph{$X$-preprojective algebra} as 
\[\Psi_H^X:=\bigoplus_{i\ge0}(\Psi_H^X)_i\ \mbox{ with }\ (\Psi_H^X)_i:=\Hom_H(X,\tau^{-i}(X)),\]
where the multiplication is defined by $g\cdot f:=\tau^{-i}(g)f$ for $f\in(\Psi_H^X)_i$ and $g\in(\Psi_H^X)_j$.
\end{definition}

By definition, $\Psi_H^X$ has a canonical structure of a $\Z$-graded algebra.
Then the \emph{preprojective algebra} of $H$ is defined as
\[\Pi=\Pi(H):=\Psi_H^H.\] 
The third equivalent definition of the preprojective algebra of $H$ is given by the tensor algebra of the $H$-bimodule $\Ext_H^1(D(H),H)$ over $H$.
Equivalence of these three definitions was first proven by Ringel \cite{R}, see also \cite{S} for a modern proof.

It follows by Gabriel's classical Theorem that $\dim_k\Pi$ is finite if and only if $H$ is of Dynkin type, and this is also equivalent to that there are only finitely many isoclasses of indecomposable $H$-modules.
In this case, $\Pi$ regarded as an $H$-module is the direct sum of all indecomposable $H$-modules.
The first aim of this paper is to study the $X$-preprojective algebra $\Psi_H^X$ for the maximum choice $X:=\Pi$ of $X$. 

\begin{definition}
The \emph{total preprojective algebra} of a representation-finite hereditary algebra $H$ is defined as
\[\Psi=\Psi(H):=\Psi_H^{\Pi}.\]
\end{definition}

If $H$ is of Dynkin type, then $\Pi$ is an additive generator of $\mod H$. Thus for each $X\in\mod H$, there exists an idempotent $e$ of $\Psi$ such that $\Psi_H^X$ is Morita equivalent to $e\Psi e$, see Proposition \ref{Moritaequpreproj}.

Our first main result gives an alternative description of the algebra $\Psi$. It is in fact isomorphic to the endomorphism algebra of a so-called 2-cluster tilting $\Pi$-module. Recall that, for a finite dimensional algebra $\Lambda$ and a positive integer $d$, we call $M\in \mod \Lambda$ \emph{$d$-cluster tilting} if
\begin{align*}
\add M&=\{X\in\mod \Lambda\mid\forall 1\le i\le d-1\ \Ext^i_\Lambda(M,X)=0\}\\
&=\{X\in\mod \Lambda\mid\forall 1\le i\le d-1\ \Ext^i_\Lambda(X,M)=0\}.
\end{align*}
This notion plays an important role in higher dimensional Auslander-Reiten theory \cite{I} as well as categorification of Fomin-Zelevinsky cluster algebras. In their study of the coordinate ring of unipotent subgroups of complex semisimple Lie groups, Geiss-Leclerc-Schr\"oer \cite{GLS} proved that the preprojective algebra $\Pi$ of a Dynkin quiver always admits a 2-cluster tilting module. It follows from \cite{IO,Y} that such a 2-cluster tilting $\Pi$-module can described as $\Pi\otimes_H\Pi$, we refer to the preliminaries for more information. One of our main results of this paper is the following.

\begin{theorem}\label{main kQ}
Let $H$ be a hereditary algebra of Dynkin type, and $\Pi$ the preprojective algebra of $H$.
\begin{enumerate}[\rm(a)]
\item We have an isomorphism of algebras
\[\Psi\simeq\End_\Pi(\Pi\otimes_H\Pi).\]
\item $\Psi$ is a 2-Auslander algebra, that is, $\gldim \Psi\le3\le\domdim \Psi$ holds.
\item Let $\Gamma:=\End_H(\Pi)$ be the Auslander algebra of $H$. Then $\Psi$ is isomorphic to the tensor algebra 
\[\Psi\simeq T_\Gamma\Hom_H(\Pi,\Pi\otimes_H\Pi_1).\]
\end{enumerate}
\end{theorem}

We illustrate Theorem \ref{main kQ} with the following figure.
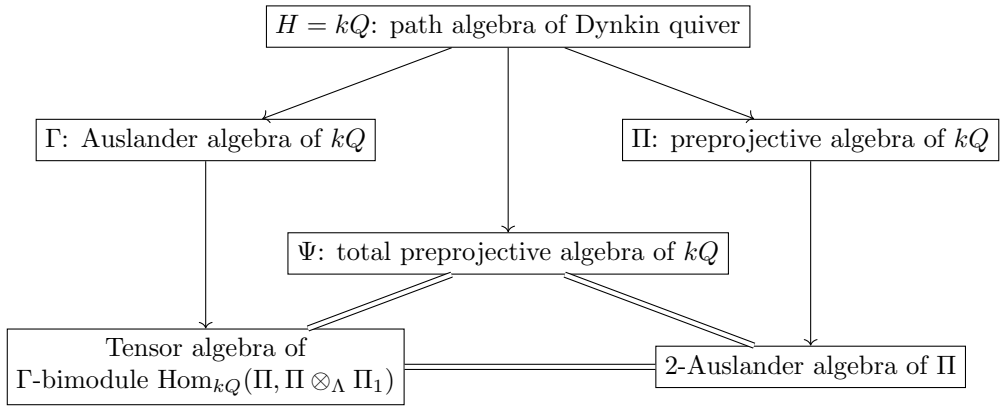
\[\begin{tikzpicture}[every text node part/.style={align=center}]
\tikzset{
    equ/.style={-,double equal sign distance},
  }
\node[draw] (vAusA) at (0,3.5) {$\Gamma$: Auslander algebra of $kQ$};
\node[draw] (vPsi) at (4,2) {$\Psi$: total preprojective algebra of $kQ$};
\node[draw] (vT) at (0,0.5) {Tensor algebra of \\$\Gamma$-bimodule $\Hom_{kQ}(\Pi,\Pi\otimes_\Lambda\Pi_1)$};
\node[draw] (vAusPi) at (8,0.5) {$2$-Auslander algebra of $\Pi$};
\node[draw] (vPi) at (8,3.5) {$\Pi$: preprojective algebra of $kQ$};
\node[draw] (v1) at (4,5) {$H=kQ$: path algebra of Dynkin quiver};
\draw  (v1) edge[->] (vPi);
\draw  (vPi) edge[->] (vAusPi);
\draw  (v1) edge[->] (vPsi);
\draw  (vPsi) edge[equ] (vAusPi);
\draw  (vAusPi) edge[equ] (vT);
\draw  (vPsi) edge[equ] (vT);
\draw  (vAusA) edge[->] (vT);
\draw  (v1) edge[->] (vAusA);
\end{tikzpicture}\]

In fact, this is a special case of our more general result as follows.
Let $d$ be a positive integer, and $\Lambda$ a finite dimensional algebra with $\gldim \Lambda\le d$. The \emph{inverse $d$-Auslander-Reiten translation} is the functor
\[\tau_d^-:=\Ext^d_\Lambda(D\Lambda,-):\mod \Lambda\to\mod \Lambda.\]
We refer to, for example, \cite{I} for more on higher Auslander-Reiten theory.

\begin{definition}
For $X\in\mod \Lambda$, we define the \emph{$X$-$(d+1)$-preprojective algebra} as 
\[\Psi_\Lambda^X:=\bigoplus_{i\ge0}(\Psi_\Lambda^X)_i\ \mbox{ with }\ (\Psi_\Lambda^X)_i:=\Hom_\Lambda(X,\tau_d^{-i}(X)),\]
where the multiplication is defined by $g\cdot f:=\tau_d^{-i}(g)f$ for $f\in(\Psi_\Lambda^X)_i$ and $g\in(\Psi_\Lambda^X)_j$.
\end{definition}

For example, the \emph{$(d+1)$-preprojective algebra} of $\Lambda$ \cite{IO} is defined as
\[\Pi=\Pi(\Lambda):=\Psi_\Lambda^\Lambda,\]
which is a $\Z$-graded algebra with $\Pi_i=(\Psi_\Lambda^\Lambda)_i$.

The algebra $\Pi$ is the 0-th cohomology of the $(d+1)$-Calabi-Yau completion of $\Lambda$ \cite{K}. Recall that a finite dimensional algebra $\Lambda$ with $\gldim \Lambda\le d$ is called \emph{$\tau_d$-finite} if $\tau_d^{-i}=0$ holds for $i\gg0$, or equivalently, $\tau_d^i=0$ for $i\gg0$. Clearly $\dim_k\Pi$ is finite if and only if $\Lambda$ is $\tau_d$-finite. 
Recall that  $\Lambda$ is called \emph{$d$-representation-finite} if there exists a $d$-cluster tilting $\Lambda$-module $M$ and $\gldim \Lambda\le d$ holds. In this case, the following assertions hold.
\begin{enumerate}[$\bullet$]
\item $\Lambda$ is $\tau_d$-finite \cite[Proposition 1.3]{I2}.
\item $\Pi$ regarded as an $\Lambda$-module gives a unique basic $d$-cluster tilting $\Lambda$-module \cite{IO}. 
\item $\Pi$ admits a $(d+1)$-cluster tilting module $\Pi \otimes_\Lambda\Pi$ \cite{IO,Y}, see Proposition \ref{cluster tilting} below.
\end{enumerate}
For the case $d=1$, the path algebra $\Lambda=kQ$ of an acyclic quiver $Q$ is $\tau_1$-finite if and only if it is $1$-representation-finite if and only if $Q$ is Dynkin.
As in the hereditary case, we introduce the following notion.

\begin{definition}
The \emph{total $(d+1)$-preprojective algebra} of a $\tau_d$-representation-finite algebra $\Lambda$ is defined as 
\[\Psi=\Psi(\Lambda):=\Psi_\Lambda^{\Pi},\]
which is a $\Z$-graded algebra with $\Psi_i=(\Psi_\Lambda^{\Pi})_i$.
\end{definition}


Our next main result is the following general form of Theorem \ref{main kQ}.

\begin{theorem}\label{main d}
Let $\Lambda$ be a finite dimensional algebra that is $\tau_d$-finite and has $\gldim \Lambda\le d$, $\Pi=\Pi(\Lambda)$ the $(d+1)$-preprojective algebra of $\Lambda$. 
\begin{enumerate}[\rm(a)]
\item We have an isomorphism of algebras
\[\Psi\simeq\End_{\Pi}(\Pi\otimes_\Lambda\Pi).\]
\item If, moreover, $\Lambda$ is $d$-representation-finite, then $\Psi$ is a $(d+1)$-Auslander algebra, that is, $\gldim \Psi\le d+2\le \domdim \Psi$ holds.
\item Let $\Gamma:=\End_A(\Pi)$. Then $\Psi$ is isomorphic to the tensor algebra 
\[\Psi\simeq T_\Gamma\Hom_\Lambda(\Pi,\Pi\otimes_\Lambda\Pi_1).\]
\end{enumerate}
\end{theorem}

Notice that, if $\Lambda$ is $d$-representation-finite, then $\Gamma$ above is a \emph{$d$-Auslander algebra} of $\Lambda$ \cite{I}.


We illustrate Theorems \ref{main d}(a)(b) with the following figure. 
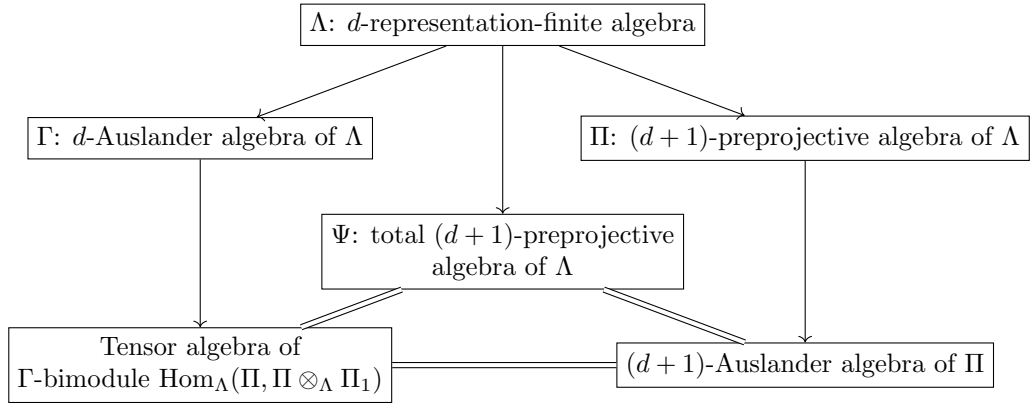
\[\begin{tikzpicture}[every text node part/.style={align=center}]
\tikzset{
    equ/.style={-,double equal sign distance},
  }
\node[draw] (vAusA) at (0,3.5) {$\Gamma$: $d$-Auslander algebra of $\Lambda$};
\node[draw] (vPsi) at (4,2) {$\Psi$: total $(d+1)$-preprojective\\algebra of $\Lambda$};
\node[draw] (vT) at (0,0.5) {Tensor algebra of \\$\Gamma$-bimodule $\Hom_\Lambda(\Pi,\Pi\otimes_\Lambda\Pi_1)$};
\node[draw] (vAusPi) at (8,0.5) {$(d+1)$-Auslander algebra of $\Pi$};
\node[draw] (vPi) at (8,3.5) {$\Pi$: $(d+1)$-preprojective algebra of $\Lambda$};
\node[draw] (v1) at (4,5) {$\Lambda$: $d$-representation-finite algebra};
\draw  (v1) edge[->] (vPi);
\draw  (vPi) edge[->] (vAusPi);
\draw  (v1) edge[->] (vPsi);
\draw  (vPsi) edge[equ] (vAusPi);
\draw  (vAusPi) edge[equ] (vT);
\draw  (vPsi) edge[equ] (vT);
\draw  (vAusA) edge[->] (vT);
\draw  (v1) edge[->] (vAusA);
\end{tikzpicture}\]

In parallel to the equivalent characterisations of the classical preprojective algebra shown by Ringel, we show that the total preprojective algebras admit a combinatorial description by quiver and relations.
Let $\Lambda$ be a $d$-representation-finite algebra with $d$-cluster tilting module $\Pi$, and $\Gamma=\End_\Lambda(\Pi)$ the $d$-Auslander algebra. We take a presentation of $\Gamma$ by quiver with relations
\[\Gamma \cong kQ_\Gamma/I_\Gamma.\]
Then the vertices of $Q_\Gamma$ are in bijection to the indecomposable direct summands $X$ of $\Pi$. Define a new quiver $\widetilde{Q}_\Gamma$ by adding to the quiver $Q_\Gamma$ a new arrow $q_X: \tau_d^-(X) \rightarrow X$ for every indecomposable non-injective summand $X$ of $\Pi$.
Locally, we have the following subquiver in $\widetilde{Q}_{\Gamma}$.
\[\begin{tikzcd}
	X \dar[swap]{a} & {\tau_d^-(X)} \lar[swap]{q_X} \dar{\tau_d^-(a)} \\
	Y & {\tau_d^-(Y)} \lar{q_Y}
\end{tikzcd}\]
For each arrow $a:X \rightarrow Y$ of $Q_\Gamma$, let
\[K\widetilde{Q}_\Gamma/(I_\Gamma)\ni r_a:=\left\{\begin{array}{ll}
aq_X-q_Y\tau_d^-(a)&\mbox{if $X$ and $Y$ are non-injective,}\\
aq_X&\mbox{if $X$ is non-injective and $Y$ is injective,}\\
0&\mbox{if $X$ is injective.}
\end{array}\right.\]

\begin{theorem}\label{mainpresentation}
Let $\Lambda$ be a $d$-representation-finite algebra with $d$-cluster tilting module $\Pi$. Then we have a $k$-algebra isomorphism
\[\Psi\simeq  k\widetilde{Q}_\Gamma/(I_\Gamma,r_a\mid a\in (Q_\Gamma)_1).\]

\end{theorem}

We remark that in the case of a hereditary algebra $\Lambda=kQ$ of Dynkin type, the combinatorial description of the total preprojective algebras coincide with the algebras studied in \cite{GLS}.


\medskip\noindent
{\bf Acknowledgements. }
This research is motivated by experiments with the computer algebra package \cite{QPA}. AC is supported by JSPS Grant-in-Aid for Scientific Research (C) 24K06666.  OI is supported by JSPS Grant-in-Aid for Scientific Research (B) 22H01113.  
RM was supported by the DFG with the project number 428999796.

\section{Preliminaries}

\subsection{On $X$-$(d+1)$-preprojective algebras}
In this section, we prove the following basic properties of $X$-$(d+1)$-preprojective algebras.

\begin{proposition} \label{Moritaequpreproj}
Let $\Lambda$ be a finite dimensional algebra with $\gldim\Lambda\le d$, and $X,Y\in\mod\Lambda$.
\begin{enumerate}[\rm(a)]
    \item If $Y$ is a direct summand of $X$, then $\Psi_{\Lambda}^Y \cong e \Psi_{\Lambda}^X e$ for an idempotent $e \in \Psi_{\Lambda}^X$.
    \item If $\add X = \add Y$, then $\Psi_{\Lambda}^X$ is Morita equivalent to $\Psi_{\Lambda}^Y$.
    \item If $Y \in \add X$, then $\Psi_{\Lambda}^Y$ is Morita equivalent to $e\Psi_{\Lambda}^Xe$ for an idempotent $e \in \Psi_{\Lambda}^X$.
\end{enumerate}
\end{proposition}

We prepare the following basic observations.

\begin{lemma} \label{Moritaequpreproj0}
Let $\C$ be an additive category, and $X,Y\in\C$.
\begin{enumerate}[\rm(a)]
    \item If $Y$ is a direct summand of $X$, then $\End_{\C}(Y) \cong e\End_{\C}(X)e$ for an idempotent $e \in \End_{\C}(X)$.
    \item If $\add X = \add Y$, then $\End_{\C}(X)$ is Morita equivalent to $\End_{\C}(Y)$.
\end{enumerate}
\end{lemma}

\begin{proof}
(a) Take $i:Y\to X$ and $p:X\to Y$ such that $pi=1_Y$. Then $e:=ip\in\End_{\C}(X)$ satisfies the desired condition.

(b) $P:=\Hom_{\C}(X,Y)$ gives a progenerator of $\End_{\C}(X)$ such that $\End_{\End_{\C}(X)}(P)=\End_{\C}(Y)$.
\end{proof}

For a finite dimensional algebra $\Lambda$ with $\gldim\Lambda\le d$, define the \emph{orbit category} $(\mod\Lambda)/\tau_d^-$ as follows.
The objects are the same as $\mod\Lambda$, and the Hom-space is defined by
\[\Hom_{(\mod\Lambda)/\tau_d^-}(X,Y):=\bigoplus_{i\ge0}\Hom_\Lambda(X,\tau_d^{-i}(Y)).\]
The composition of $f\in\Hom_\Lambda(X,\tau_d^{-i}(Y))\subset\Hom_{(\mod\Lambda)/\tau_d^-}(X,Y)$ and $g\in\Hom_\Lambda(Y,\tau_d^{-j}(Z))\subset\Hom_{(\mod\Lambda)/\tau_d^-}(Y,Z)$ is $\tau_d^{-i}(g)\circ f\in\Hom_\Lambda(X,\tau_d^{-i-j}(Z))\subset\Hom_{(\mod\Lambda)/\tau_d^-}(X,Z)$.

The following assertion is clear from the definition.


\begin{lemma}\label{orbit category}
For each $X\in\mod\Lambda$, we have an isomorphism $\Psi_\Lambda^X\simeq\End_{(\mod\Lambda)/\tau_d^-}(X)$ of algebras.
\end{lemma}


Now we are ready to prove Proposition \ref{Moritaequpreproj}.

\begin{proof}[Proof of Proposition \ref{Moritaequpreproj}]
Let $\C:=(\mod\Lambda)/\tau_d^-$. Then the claims follow from Lemmas \ref{Moritaequpreproj0} and \ref{orbit category}.
\end{proof}

\subsection{Cluster tilting for higher preprojective algebras}

Throughout this section, let $\Lambda$ be a $d$-representation-finite algebra and $\Pi$ be its $(d+1)$-preprojective algebra for some $d\ge 1$.
The aim of this section is to explain that results in \cite{IO,Y} show that $\Pi\otimes_\Lambda\Pi$ is a $d$-cluster tilting $\Pi$-module.

Recall that $\Pi$ is a $\Z$-graded algebra with $\Pi_i=\tau_d^{-i}(\Lambda)$.
Let $\mod^\Z\Pi$ be the category of finitely generated $\Z$-graded $\Pi$-modules, and $\underline{\mod}^\Z\Pi$ its stable category.
For $M\in\mod^{\Z}\Pi$ and $j\in\Z$, we have a subobject $M_{\ge j}:=\bigoplus_{i\ge j}M_j$ of $M$ and a factor object $M_{\le j}:=\bigoplus_{i\le j}M_i$ of $M$.

\begin{proposition}[cf.\ \cite{IO,Y}]\label{tilting}
The stable category $\underline{\mod}^\Z \Pi$ admits tilting objects
\[T:= \bigoplus_{i\ge 0} \Pi(i)_{\le0}\ \mbox{ and }\ U:=\bigoplus_{i\ge 1} \Pi(i)_{\ge 0}\simeq\Pi\otimes_\Lambda\Pi,\]
where we define the $\Z$-grading on $\Pi\otimes_\Lambda\Pi$ by $(\Pi\otimes_\Lambda\Pi)_i:=\Pi\otimes_\Lambda\Pi_i$.
\end{proposition}

\begin{proof}
By \cite[(3.1)]{Y}, $\underline{\mod}^\Z \Pi$ admits a tilting object $T := \bigoplus_{i\ge 0} \Pi(i)_{\le0}$.
Applying the degree shift functor $(1)$ and taking its syzygy, we obtain another tilting object
\[
U:=\Omega T(1) = \bigoplus_{i\ge 1} \Pi(i)_{\ge 0} = \bigoplus_{i\ge 1} \Pi_{\ge i}(i).\]
Since $\Pi\simeq\bigoplus_{i\ge0}\Pi_i$ as $\Lambda$-modules and $\Pi_i\otimes_\Lambda\Pi\simeq\Pi_{\ge i}(i)$ as $\Z$-graded $\Pi$-modules, we have 
\[
\Pi\otimes_\Lambda \Pi\simeq\bigoplus_{i\ge 0} \Pi_{i}\otimes_\Lambda \Pi\simeq\bigoplus_{i\ge 0} \Pi_{\ge i}(i) =\Pi\oplus U\]
in $\mod^{\Z}\Pi$. Thus the assertion follows.
\end{proof}

Let $\underline{\Gamma}:=\underline{\End}_\Pi^\Z(T)$ be the endomorphism ring of $T$ in $\underline{\mod}^\Z\Pi$. Then we have a triangle equivalence $\underline{\mod}^\Z \Pi\simeq \Db(\mod \underline{\Gamma})$ sending $T$ to $\underline{\Gamma}$ \cite[Corollary 4.7]{Y}.
By \cite[Theorem 4.6]{Y}, $\underline{\Gamma}$ is isomorphic to the stable $d$-Auslander algebra $\underline{\End}_\Lambda(\Pi)$ of $\Lambda$, which justifies our choice of notation. 
Thus $\gldim\underline{\Gamma}\le d+1$ holds by \cite[Theorem 2.23]{IO}.

Let $\CC_{d+1}(\underline{\Gamma})$ be the $(d+1)$-cluster category of $\underline{\Gamma}$, and $\pi:\Db(\mod\underline{\Gamma})\to\CC_{d+1}(\underline{\Gamma})$ the canonical functor \cite{Am,G}.
By \cite[Proposition 4.14, Theorem 4.15]{IO}, we have a triangle equivalence $\underline{\mod}\,\Pi\simeq\CC_{d+1}(\underline{\Gamma})$ making the following diagram commutative.
\[\xymatrix@R1.5em{
\underline{\mod}^{\Z}\Pi\ar[r]\ar[d]^{\rm forget}&\Db(\mod\underline{\Gamma})\ar[d]^\pi\\
\underline{\mod}\,\Pi\ar[r]&\CC_{d+1}(\underline{\Gamma}).
}\]
In particular, $\CC_{d+1}(\underline{\Gamma})$ is Hom-finite, and hence $\CC_{d+1}(\underline{\Gamma})$ is a $(d+1)$-Calabi-Yau triangulated category with  
 $(d+1)$-cluster tilting object $\pi \underline{\Gamma}$ \cite[Theorem 4.9]{G}.
Consequently, we obtain the following observation.



\begin{proposition}[cf.\ \cite{IO}]\label{cluster tilting}
The stable category $\underline{\mod}\,\Pi$ admits $(d+1)$-cluster tilting objects $T$ and $U$ given in Proposition \ref{tilting}.  In particular, $\Pi\otimes_\Lambda\Pi$ is a $(d+1)$-cluster tilting $\Pi$-module.
\end{proposition}

\begin{proof}
Since the equivalence $\underline{\mod}\,\Pi\simeq\CC_{d+1}(\underline{\Gamma})$ sends $T$ to $\pi\underline{\Gamma}$, $T$ is a $(d+1)$-cluster tilting object in $\underline{\mod}\,\Pi$.  Since $U=T[-1]$, it is also $(d+1)$-cluster tilting in $\underline{\mod}\,\Pi$, which means that $U$ is $(d+1)$-cluster tilting as $\Pi$-module.
\end{proof}

\section{Extended tensor algebras of bimodules}

Let $\Lambda$ always denote a finite dimensional $k$-algebra over a field $k$. Modules are finitely generated right modules unless otherwise stated.
Let $M$ be an $\Lambda$-bimodule.
Recall that the \emph{tensor algebra} $T_\Lambda(M)$ of $M$ over $\Lambda$ is defined as the algebra
\[T_\Lambda(M):=\bigoplus\limits_{i \geq 0}^{}{M^{\otimes_{\Lambda} i}}\]
with the canonical multiplication.
We denote the $i$-th graded piece of this tensor algebra by $T_\Lambda(M)_i:=M^{\otimes_{\Lambda} i}$.
We call $M$ \emph{nilpotent} if $M^{\otimes_{\Lambda} i} =0$ for some $i>0$.

\begin{definition}
Let $\Lambda$ be a ring, $M$ a nilpotent $\Lambda$-bimodule, and $T=T_\Lambda(M)$ the tensor algebra. The \emph{extended tensor algebra} is defined as the ring
\[U_\Lambda(M):=\bigoplus\limits_{i \geq 0}^{}{\Hom_\Lambda(T,T \otimes_\Lambda T_i)}\]
with the following multiplication for 
$f \in \Hom_\Lambda(T,T \otimes_\Lambda T_i)$ and $g \in \Hom_\Lambda(T,T \otimes_\Lambda T_j)$,
\begin{align}\label{define gf}
g \cdot f:= [T \xrightarrow{f} T \otimes_\Lambda T_i \xrightarrow{g \otimes_\Lambda \id_{T_i}} T \otimes_\Lambda T_j \otimes_\Lambda T_i = T \otimes_\Lambda T_{i+j}].
\end{align}
Note that $U_\Lambda(M)$ is graded with $U_\Lambda(M)_i=\Hom_\Lambda(T,T\otimes_\Lambda T_i)$.
\end{definition}


The main result in this section is the following theorem:
\begin{theorem} \label{tensormainresult}
Let $\Lambda$ be a ring, and $M$ a nilpotent $\Lambda$-bimodule which is finitely generated as a $\Lambda$-module. Then we have an isomorphism of $k$-algebras:
$$\End_{T_\Lambda(M)}(T_\Lambda(M) \otimes_\Lambda T_\Lambda(M)) \cong U_\Lambda(M).$$
\end{theorem}
\begin{proof}
For simplicity, in this proof, we denote the algebra $T_\Lambda(M)$ by $T$, $T_\Lambda(M)_i$ by $T_i$ and $\otimes_\Lambda$ by $\otimes$. 
For any $(T,\Lambda)$-module $X$ and $T$-bimodule $Y$, we have a canonical isomorphism of $T$-bimodules 
\[\Hom_\Lambda(X,Y)\simeq\Hom_T(X\otimes T,Y).\]
In particular, we have a canonical isomorphism of $T$-bimodules
\begin{align}\label{regard End as Hom}
\alpha:\Hom_\Lambda(T, T \otimes T)\cong\Hom_T(T \otimes T, T \otimes T)=\End_T(T \otimes T).
\end{align}
Since $M$ is nilpotent and finitely generated as a $\Lambda$-module, $T$ is also finitely generated as a $\Lambda$-module. Thus we have
\[
\Hom_\Lambda(T, T \otimes T) \cong \Hom_\Lambda(T, \bigoplus\limits_{i \geq 0}^{} (T \otimes T_i))=\bigoplus\limits_{i \geq 0}^{}{\Hom_\Lambda(T,T \otimes T_i)}=U_\Lambda(M).
\]
It remains to show that the isomorphism is compatible with multiplication.
Fix $f\in\Hom_\Lambda(T,T\otimes T_i)\subset\Hom_\Lambda(T,T\otimes T)$ and $g\in\Hom_\Lambda(T,T\otimes T_j)\subset\Hom_\Lambda(T,T\otimes T)$, and consider $g\cdot f\in\Hom_\Lambda(T,T\otimes T_{i+j})\subset\Hom_\Lambda(T,T\otimes T)$ given in \eqref{define gf}.
Then we have the following, where we write $1$ for the respective identity, and $m$ for the respective multiplication map. 
\begin{align*}
\alpha(f)&=[T\otimes T\xrightarrow{f\otimes 1}T\otimes T_i\otimes T\xrightarrow{1\otimes m}T\otimes T],\\
\alpha(g)&=[T\otimes T\xrightarrow{g\otimes 1}T\otimes T_j\otimes T\xrightarrow{1\otimes m}T\otimes T],\\
\alpha(g \cdot f)&=[T\otimes T\xrightarrow{g\cdot f\otimes 1}T\otimes T_{i+j}\otimes T\xrightarrow{1\otimes m}T\otimes T]\\
&\stackrel{\eqref{define gf}}{=}[T\otimes T\xrightarrow{f\otimes 1} T \otimes T_i \otimes T\xrightarrow{g \otimes 1\otimes 1} T \otimes T_j \otimes T_i\otimes T = T \otimes T_{i+j}\otimes T\xrightarrow{1\otimes m}T\otimes T].
\end{align*}
The following commutative diagram shows $\alpha(g)\alpha(f)=\alpha(g\cdot f)$ as desired.
\[\begin{tikzcd}
[column sep=.6cm,row sep=.3cm]
\tikzset{
    equ/.style={-,double equal sign distance},
  }
	{T \otimes T} && {T \otimes T_i \otimes T} && {T \otimes T}&& {T \otimes T_j \otimes T} && {T \otimes T} \\
	&&&& {T \otimes T_j \otimes T_i \otimes T} \\
	\\
	{T \otimes T } && {T \otimes T_i \otimes T} && {T \otimes T_j \otimes T_i \otimes T} && {T \otimes T_{i+j} \otimes T} && {T \otimes T}
	\arrow["{g \otimes 1}", from=1-5, to=1-7]
	\arrow["{f \otimes 1}", from=1-1, to=1-3]
	\arrow[equal, from=1-1, to=4-1]
	\arrow["{1\otimes m}", from=1-3, to=1-5]
	\arrow["{g \otimes 1 \otimes 1}"', from=1-3, to=2-5]
	\arrow[equal, from=1-3, to=4-3]
	\arrow["{1 \otimes m}", from=1-7, to=1-9]
	\arrow[equal, from=1-9, to=4-9]
	\arrow["{1 \otimes 1 \otimes m}"', from=2-5, to=1-7]
	\arrow[equal, from=2-5, to=4-5]
	\arrow["{f \otimes 1}", from=4-1, to=4-3]
	\arrow["{g \otimes 1 \otimes 1}", from=4-3, to=4-5]
	\arrow[equal, from=4-5, to=4-7]
	\arrow["1\otimes m", from=4-7, to=4-9]
\end{tikzcd}\qedhere\]
\end{proof}

For an additive category $\CC$, an object $C\in\CC$ and a covariant functor $F:\CC\to\Ab$, by applying Yoneda's Lemma we obtain bijections

\begin{equation} \label{yoneda}
F(C)\simeq\Hom(\CC(C,-),F)\simeq \CC(-,C)\otimes_\CC F.
\end{equation}

\begin{proposition}\label{tensormainresult2}
We have an isomorphism of rings:
\[U_\Lambda(M)\cong T_{U_\Lambda(M)_0}(U_\Lambda(M)_1).\]
\end{proposition}

\begin{proof}
Let $T=T_\Lambda(M)$, $T_i=M^{\otimes_\Lambda i}$ and $U_i:=U_\Lambda(M)_i$ for $i\ge0$.  For $i\ge1$, it suffices to show that the natural map $U_1\otimes_{U_0}U_i\to U_{i+1}$ is an isomorphism.
Now we consider the category $\CC:=\add T_\Lambda \subset \mod\Lambda$ and its object $C:=T\otimes_\Lambda T_1\in\CC$. Applying \eqref{yoneda} to the functor $F:=\Hom_\Lambda(T,-\otimes_AT_i):\CC\to\Ab$,
we have
\[
U_{i+1}=\Hom_\Lambda(T,T\otimes_\Lambda T_{i+1})=F(C)\simeq 
\Hom_\Lambda(T,T\otimes_\Lambda T_1)\otimes_{\End_\Lambda(T)}\Hom_\Lambda(T,T\otimes_\Lambda T_i).\]
The right-hand space is precisely the $(i+1)$-st graded piece of the tensor algebra in the claim.
This is precisely the desired isomorphism.
\end{proof}

Now we are able to prove Theorem \ref{main d}.

\begin{proof}[Proof of Theorem \ref{main d}]
(a) Let $\Lambda$ be a $\tau_d$-finite algebra.  Then we have an $\Lambda^e$-module $M:=\Ext^d_\Lambda(D\Lambda,\Lambda)$ and $T_\Lambda(M)=\Pi$, see for example \cite[Section 2]{IO}.  Since $\Pi \otimes_\Lambda \Pi_1 = \tau_d^{-1}(\Pi)$, we have $U_\Lambda(M)=\Psi_\Lambda^\Pi=\Psi$ by definition. Thus, the claim follows from Theorem \ref{tensormainresult}.

(b) Since $\Pi \otimes_\Lambda\Pi $ is a $(d+1)$-cluster tilting $\Pi $-module 
by Proposition \ref{cluster tilting}, its endomorphism algebra-- which is isomorphic to $\Psi$ by (a) -- is a $(d+1)$-Auslander algebra.

(c) This is immediate from (a) and Proposition \ref{tensormainresult2}.
\end{proof}

\section{Presentation of some tensor algebras}

Now we consider the following general setting.
\begin{enumerate}[\rm$\bullet$]
\item Let $\Lambda=kQ/I$ be a finite dimensional algebra with admissible ideal $I$.
\item Let $f\in\Lambda$ be an idempotent, and $\phi:\Lambda\to f\Lambda f$ a morphism of $k$-algebras such that, for each $i\in Q_0$, $\phi(e_i)$ is either zero or $e_{\phi(i)}$ for some $\phi(i)\in Q_0$.
\end{enumerate}

In this case, $f=\phi(1)=\sum_{i\in S}e_{\phi(i)}$ holds for $S:=\{i\in Q_0\mid \phi(e_i)\neq0\}$.

\begin{proposition}\label{presentation}
Assume that $\phi:\Lambda\to f\Lambda f$ is a morphism of algebras satisfying the aforementioned assumption.
Consider the $\Lambda$-bimodule $f\Lambda$ with the usual right action the left action induced by $\phi$. 
Then we have an isomorphism of algebras
\[T_\Lambda(f\Lambda)\simeq k\widetilde{Q}/(I,r_a\mid a\in Q_1),\]
where $\widetilde{Q}:=Q\sqcup\{ q_i:\phi(i)\to i\mid i\in S\}$ and for each arrow $a:i\to j$ of $Q$,
\[k\widetilde{Q}/(I)\ni r_a:=\left\{\begin{array}{ll}aq_i-q_j\phi(a)&i,j\in S,\\
aq_i&i\in S,\ j\notin S,\\
0&i\notin S.
\end{array}\right.\]
\end{proposition}

\begin{proof}
Consider the projective $\Lambda^e$-module and its element
\[P:=\bigoplus_{i\in S}\Lambda e_i\otimes_ke_{\phi(i)}\Lambda\ni q'_i:=e_i\otimes e_{\phi(i)}\ \mbox{ for each }\ i\in S.\]
Then the tensor algebra $T_\Lambda(P)$ has a presentation
\begin{equation}\label{presentation}
    \psi:k\widetilde{Q}/(I)\simeq T_\Lambda(P),
\end{equation}
where $\psi(q_i):=q'_i$ for each $i\in S$.
We have a morphism of $\Lambda^e$-modules
\[\alpha:P\to f\Lambda,\ x\otimes y\mapsto \phi(x)y.\]

For each $i\in Q_0\setminus S$, let $q'_i:=0\in P$. For each arrow $a:i\to j$ of $Q$, let
\[r'_a:=
aq'_i-q'_j\phi(a)\in P.\]
Let $P'$ be the $\Lambda^e$-submodule of $P$ generated by all $r'_a$.

We claim that $\alpha$ induces an isomorphism of $\Lambda^e$-modules:
\begin{equation}\label{fL}
\alpha:P/P'\simeq f\Lambda.
\end{equation}
It suffices to show $P'=\Ker\,\alpha$.
Since $\alpha(r'_a)=0$ holds,
we have $P'\subset\Ker\,\alpha$. To prove the reverse inclusion, consider the $\Lambda$-submodule $V:=\bigoplus_{i\in S}ke_i\otimes_ke_{\phi(i)}\Lambda\subset P$.
Since $\alpha|_V:V\to f\Lambda$ is an isomorphism, we have $P=\Ker\,\alpha\oplus V$ as $\Lambda$-modules.
It suffices to prove $P=P'+V$.
As a $\Lambda$-module, $P$ is generated by $a_m\cdots a_1q'_i$ for all paths $a_m\cdots a_1:i\to j$ in $Q$ such that $i\in S$. 
Thus it suffices to show $a_m\cdots a_1q'_i\in P'+V$. Let $i_\ell:=s(a_\ell)$ for $1\le\ell\le m$ and $i_{m+1}:=j$. Then $i_1=i$ and we have
\begin{align*}
a_m\cdots a_1q'_i&=\sum_{\ell=1}^ma_m\cdots a_{\ell+1}(a_\ell q'_{i_\ell}-q'_{i_{\ell+1}}\phi(a_\ell))\phi(a_{\ell-1}\cdots a_1)+q'_j\phi(a_m\cdots a_1)\\
&=\sum_{\ell=1}^ma_m\cdots a_{\ell+1}r'_{a_\ell }\phi(a_{\ell-1}\cdots a_1)+q'_j\phi(a_m\cdots a_1)\in P'+V.
\end{align*}
Thus we obtain $P=P'+V$ and hence $P'=\Ker\,\alpha$, as desired.


Consequently, we have $k$-algebra isomorphisms
\[T_\Lambda(f\Lambda )\stackrel{\eqref{fL}}{\simeq} T_\Lambda(P/P')\simeq T_\Lambda(P)/(r'_a\mid a\in Q_1)\stackrel{\eqref{presentation}}{\simeq} k\widetilde{Q}/(I,r_a\mid a\in Q_1),\]
where the last isomorphism follows from $\psi(r_a)=r'_a$.
\end{proof}

Now we are ready to prove Theorem \ref{mainpresentation}.

\begin{proof}[Proof of Theorem \ref{mainpresentation}]

Let $f$ be the idempotent of $\Gamma$ corresponding to the direct sum of indecomposable non-projective direct summands of the $\Lambda$-module $\Pi$, and let $\phi:\Gamma\to f\Gamma f$ be the morphism of $k$-algebras given by the $d$-Auslander-Reiten translation $\tau_d^-:\mod\Lambda\to\add f\Pi\subset\mod\Pi$.  The claim then follows immediate from Proposition \ref{presentation} with $S$ being the set of indecomposable non-injective $\Lambda$-module.
\end{proof}

\section{Examples}

\subsection{Hereditary algebras}

Let $Q$ be a Dynkin quiver and $H=kQ$.
The structure of the Auslander-Reiten quiver ${\rm AR}(H)$ of $H$ is well-known.
Let $\Z Q$ be the translation quiver with the set of vertices $\Z\times Q_0$ and two kinds of arrows $(i,a):(i,s(a))\to(i,t(a))$ and $(i,a^*):(i,t(a))\to(i+1,s(a))$ for each $(i,a)\in \Z\times Q_1$.
This gives the Auslander-Reiten quiver of the bounded derived category of $H$, and then ${\rm AR}(H)$ can be regarded as a full subquiver.
Moreover, the quiver $Q_\Gamma=(Q_{\Gamma,0},Q_{\Gamma,1})$ of the Auslander algebra $\Gamma$ of $H$ coincides with ${\rm AR}(H)$. Let $Q_{\Gamma,0}^{\rm ni}$ the set of vertices of $Q_\Gamma$ corresponding to indecomposable non-injective $H$-modules. 
Then we have a presentation
\[\Gamma\simeq kQ_\Gamma/(\sum_{a\in(\Z Q)_1}((i,a^*)(i,a)-(i+1,a)(i,a^*)))\]
For an indecomposable non-injective $H$-module $X$, let $m_X$ be the mesh relation starting at $X$, i.e.
\[
m_X = \sum_{a\in(\Z Q)_1}e_{\tau^-X}((i,a^*)(i,a)-(i+1,a)(i,a^*))e_X.
\]
Let $\widetilde{Q}_\Gamma$ be the quiver obtained by adding new arrows $q_X:\tau^-X\to X$ to $Q_\Gamma$ for each indecomposable non-injective $H$-module $X$.
By Theorem \ref{mainpresentation}, we have a presentation of the total preprojective algebra of $H$
\[\Psi\simeq k\widetilde{Q}_\Gamma/(m_X,r_a\mid X\in Q_{\Lambda,0}^{\rm ni},\,a\in Q_{\Gamma,1}),\]
where, for each arrow $a:X\to Y$ of $Q_\Gamma$,
\[r_a:=\left\{\begin{array}{ll}
aq_X-q_Y\tau^-(a)&\mbox{if $X$ and $Y$ are non-injective,}\\
aq_X&\mbox{if $X$ is non-injective and $Y$ is injective,}\\
0&\mbox{if $X$ is injective.}
\end{array}\right.\]
This description coincides with that of Geiss-Leclerc-Schr\"oer \cite[Theorem 1]{GLS} (after swapping the role of non-projective with non-injectives as they use a dual convention).

For example, let $Q$ be the quiver type $D_4$ shown on the left-hand side below. Then the quiver $Q_\Gamma$ of the Auslander algebra is the one on the right-hand side, where $d:=a^*$, $e:=b^*$ and $f:=c^*$
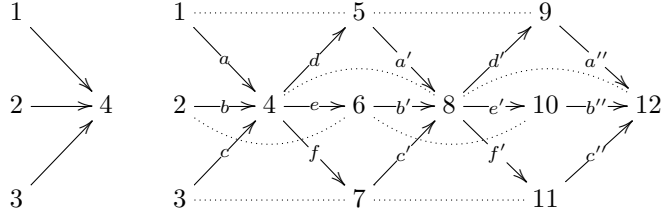
\[\xymatrix@C0.8cm@R0.8cm{1\ar[rd]\\ 2\ar[r]&4\\ 3\ar[ru]}
\ \ \ \ \ \xymatrix@C0.8cm@R0.8cm{
1\ar[rd]|a&&5\ar[rd]|{a'}\ar@{.}[ll]&&9\ar[rd]|{a''}\ar@{.}[ll]\\
2\ar[r]|b&4\ar[ru]|d\ar[r]|e\ar[rd]|f&6\ar[r]|{b'}\ar@{.}@/^5mm/[ll]&8\ar[ru]|{d'}\ar[r]|{e'}\ar[rd]|{f'}\ar@{.}@/^-5mm/[ll]&10\ar[r]|{b''}\ar@{.}@/^5mm/[ll]&12\ar@{.}@/^-5mm/[ll]&\\
3\ar[ru]|c&&7\ar[ru]|{c'}\ar@{.}[ll]&&11\ar[ru]|{c''}\ar@{.}[ll]&&
}\]
The relations of $\Gamma$ are $da, eb, fc, a'd+b'e+c'f, d'a', e'b', f'c', a''d'+b''e'+c''f'$. The quiver of the total preprojective algebra $\Psi$ is
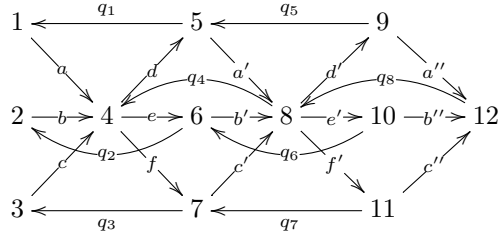
\[\xymatrix@C0.8cm@R0.8cm{
1\ar[rd]|a&&5\ar[rd]|{a'}\ar[ll]_{q_1}&&9\ar[rd]|{a''}\ar[ll]_{q_5}\\
2\ar[r]|b&4\ar[ru]|d\ar[r]|e\ar[rd]|f&6\ar[r]|{b'}\ar@/^5mm/[ll]|{q_2}&8\ar[ru]|{d'}\ar[r]|{e'}\ar[rd]|{f'}\ar@/^-5mm/[ll]|{q_4}&10\ar[r]|{b''}\ar@/^5mm/[ll]|{q_6}&12\ar@/^-5mm/[ll]|{q_8}\\
3\ar[ru]|c&&7\ar[ru]|{c'}\ar[ll]^{q_3}&&11\ar[ru]|{c''}\ar[ll]^{q_7}
}\]
where the additional relations are
$$\begin{array}{l}aq_1-q_4a', bq_2-q_4b', cq_3-q_4c', dq_4-q_5d', eq_4-q_6e', fq_4-q_7f', a'q_5-q_8a'', b'q_6-q_8b'', c'q_7-q_8c'',\\
d'q_8,e'q_8,f'q_8.
\end{array}$$

\subsection{Higher Auslander algebras of type $A$}
Fix positive integers $d$ and $n$. We define the quiver $Q:=Q^{(d,n)}$ 
with the set $Q_0$ of vertices and the set $Q_1$ of arrows by
\begin{eqnarray*}
Q_0&:=&\{x=(x_1,x_2,\ldots,x_{d+1})\in\Z_{\ge0}^{d+1}\ |\ \sum_{i=1}^{d+1}x_i=n-1\},\\
Q_1&:=&\{a_{x,i}:x\to x+f_i\ |\ 1\le i\le d+1,\ x,x+f_i\in Q_0\}
\end{eqnarray*}
where $f_i$ denotes the vector $f_i:=(0,\ldots,0,\stackrel{i}{-1},\stackrel{i+1}{1},0,\ldots,0)$ for $1\le i\le d$ and $f_{d+1}:=(\stackrel{1}{1},0,\ldots,0,\stackrel{d+1}{-1})$.
The following picture shows $Q^{(d,3)}$ for $d=1,2,3$.
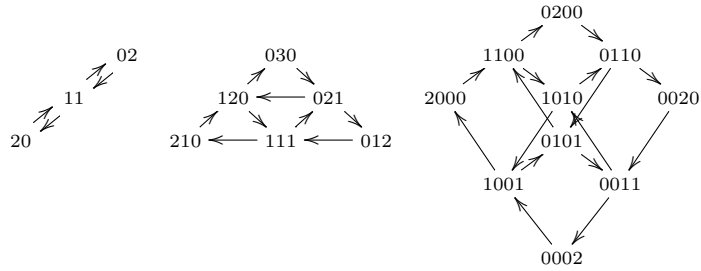
\[\xymatrix@C0cm@R0.2cm{
&&&&&
  &&&&&&&
    &&{\scriptstyle 0200}\ar[rd]&&\\
&&&&{\scriptstyle 02}\ar@<.5ex>[lld]&
  &&&{\scriptstyle 030}\ar[rd]&&&&
    &{\scriptstyle 1100}\ar[ru]\ar[rd]&&{\scriptstyle 0110}\ar[rd]\ar[ldd]&\\
&&{\scriptstyle 11}\ar@<.5ex>[rru]\ar@<.5ex>[lld]&&
  &&&{\scriptstyle 120}\ar[ru]\ar[rd]&&{\scriptstyle 021}\ar[ll]\ar[rd]&&&
    {\scriptstyle 2000}\ar[ru]&&{\scriptstyle 1010}\ar[ru]\ar[ldd]&&{\scriptstyle 0020}\ar[ldd]\\
{\scriptstyle 20}\ar@<.5ex>[rru] &&&&&
  &{\scriptstyle 210}\ar[ru]&&{\scriptstyle 111}\ar[ll]\ar[ru]&&{\scriptstyle 012}\ar[ll]&&
    &&{\scriptstyle 0101}\ar[rd]\ar[luu]&&\\
&&&&&
  &&&&&&&
    &{\scriptstyle 1001}\ar[ru]\ar[luu]&&{\scriptstyle 0011}\ar[luu]\ar[ldd]&\\
&&&&&
   &&&&\\
&&&&&
   &&&&&&&
    &&{\scriptstyle 0002}\ar[luu]&&
}\]
We define the $k$-algebra
\[\Pi^{(d,n)}:=kQ^{(d,n)}/I^{(d,n)},\]
where $I^{(d,n)}$ is the ideal defined by the following relations which we denote by $r_{x,i,j}$:
For any $x\in Q_0$ and $1\le i\neq j\le d+1$ satisfying $x+f_i,x+f_i+f_j\in Q_0$, let 
\[r_{x,i,j}:=\left\{\begin{array}{ll}
(x\xrightarrow{a_{x,i}}x+f_i\xrightarrow{a_{x+f_i,j}}x+f_i+f_j)-(x\xrightarrow{a_{x,j}}x+f_j\xrightarrow{a_{x+f_j,i}}x+f_i+f_j)&\mbox{if}\ x+f_j\in Q_0,\\
(x\xrightarrow{a_{x,i}}x+f_i\xrightarrow{a_{x+f_i,j}}x+f_i+f_j)&\mbox{otherwise}.
\end{array}
\right.\]
We denote by $\check{Q}^{(d,n)}$ the quiver obtained from $Q^{(d,n)}$ by removing all arrows of the form $a_{x,d+1}$.
We define the \emph{higher Auslander algebra of type $A_n$} by
\[\Lambda^{(d,n)}:=k\check{Q}^{(d,n)}/\check{I}^{(d,n)},\]
where $\check{I}^{(d,n)}$ is the ideal generated by all relations $r_{x,i,j}$ satisfying $i,j\neq d+1$. Then we have an isomorphism 
 \[\Lambda^{(d,n)}\simeq\Pi^{(d,n)}/\overline{I}^{(d,n)},\]
where $\overline{I}^{(d,n)}$ is the ideal generated by all arrows of the form $a_{x,d+1}$. Let
\[\Psi^{(d,n)}:=kQ^{(d,n)}/J^{(d,n)},\]
where $J^{(d,n)}$ is the ideal generated by all $r_{x,i,j}$, except for $r_{x,i,d+1}$ satisfying $x_{d+1}=0$. 

\begin{proposition}
For positive integers $d$ and $n$, the following assertions hold.
\begin{enumerate}[\rm(a)]
\item $\Lambda^{(d,n)}$ is $d$-representation finite algebra whose $d$-Auslander algebra is $\Lambda^{(d+1,n)}$.
\item The $(d+1)$-preprojective algebra of $\Lambda^{(d,n)}$ is $\Pi^{(d,n)}$.
\item The total $(d+1)$-preprojective algebra of $\Lambda^{(d,n)}$ is $\Psi^{(d+1,n)}$.
\end{enumerate}
\end{proposition}

\begin{proof}
(a)(b) See \cite{I2,IO}.

(c) Let $\Lambda:=\Lambda^{(d,n)}$. By (a) and (b), the $d$-Auslander-algebra of $\Lambda$ is $\Gamma:=\Lambda^{(d+1,n)}$, and the $(d+1)$-preprojective algebra of $\Lambda$ is $\Pi:=\Pi^{(d,n)}$. Let $\Psi$ be the total $(d+1)$-preprojective algebra of $\Lambda$.

The quiver of the $d$-cluster tilting subcategory $\add\Pi$ of $\mod\Lambda$ is given by $\check{Q}^{(d+1,n)}$. We consider the $d$-Auslander-Reiten translation $\tau_d^-:\add\Pi\to\add\Pi$.
We identify the vertices of $\check{Q}^{(d+1,n)}$ with indecomposable modules of $\add\Pi$. Then the indecomposable injective (respectively, projective) $\Lambda$-modules are the vertices $x$ satisfying $x_1=0$ (respectively, $x_{d+2}=0$), and moreover $\tau_d^-(x)=x-f_{d+2}$ holds for each indecomposable non-injective object $x$ in $\add\Pi$.
By Theorem \ref{mainpresentation}, the quiver of $\Psi$ is given by adding new arrows $q_x:x-f_{d+2}\to x$ to $\check{Q}^{(d+1,n)}$, which is precisely $Q^{(d+1,n)}$ by identifying $q_x$ with $a_{x-f_{d+2},d+2}$.

The relations $r_{x,i,j}$ when both $i,j\neq d+1$ are the generating relations of $\Lambda^{(d+1,n)}$.  Hence, to obtain the desired isomorphism $\Psi\simeq kQ^{(d,n)}/J^{(d,n)}$, it remains to show that the relations $r_{x,d+2,i}$ for all $x$ with $x_{d+2}\neq 0$ coincide with the $r_a$'s of Theorem \ref{mainpresentation}.
Indeed, first note that
\[
r_{x,d+2,i} = \begin{cases}
a_{x+f_{d+2},i}a_{x,d+2}-a_{x+f_i,d+2}a_{x,i}, & \text{if }x+f_{d+2}, x+f_i\in Q^{(d+1,n)},\\
a_{x+f_{d+2},i}a_{x,d+2}, & \text{if }x+f_i\notin Q^{(d+1,n)}, x+f_{d+2}\in Q^{(d+1,n)},\\
-a_{x+f_i,d+2}a_{x,i}, & \text{if }x+f_i\in Q^{(d+1,n)}, x+f_{d+2}\notin Q^{(d+1,n)},\\
0, & \text{otherwise.}
\end{cases}
\]
We need to exclude the third case, i.e. when $x+f_{d+2}\notin Q^{(d+1,n)}$, which is equivalent to having $x_{d+2}=0$.
Now $x+f_{d+2}\in Q^{(d+1,n)}$ is equivalent to $x+f_{d+2}$ being non-injective.
If $x+f_i+f_{d+2}\in Q^{(d+1,n)}$, then $x+f_i \notin Q^{(d+1,n)}$ is equivalent to $x+f_i+f_{d+2}$ being injective.  
Hence, the first and second case above gives all the $r_a$'s of Theorem \ref{mainpresentation} with $a=a_{x+f_{d+2},i}$.

\end{proof}

We give an explicit example.
Let $\Lambda=\Lambda^{(2,3)}$ be the Auslander algebra of the $A_3$ quiver $k Q^{(1,3)}=\Lambda^{(1,3)}$. Then $\Lambda$ is given by the following quiver-and-relations:
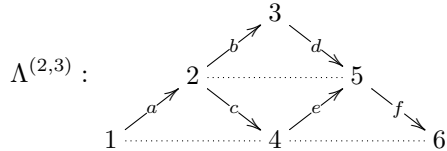
\[
\Lambda^{(2,3)}: \vcenter{\xymatrix@R=0.4cm@C=0.7cm{&& 3\ar[rd]|{d} && \\ &2\ar[rd]|{c}\ar[ru]|{b}&&5\ar[rd]|{f} \ar@{.}[ll]& \\ 1\ar[ru]|{a}&&4\ar[ru]|{e} \ar@{.}[ll] &&6 \ar@{.}[ll] }}
\]
Here, the dotted lines represent the mesh relations $m_1=ca, m_4=fe, m_2=db-ec$. 

Let $P_i$, $S_i$ and $I_i$ be projective, simple and injective $\Lambda$-modules corresponding to the vertex $i$ respectively.
Then the $2$-cluster tilting $\Lambda$-module is given by $\Lambda\oplus\tau_2^{-1}(\Lambda)\oplus\tau_2^{-2}(\Lambda)$, whose basic part is given by $P_1\oplus P_2\oplus P_3\oplus P_4\oplus P_5\oplus P_6\oplus S_4\oplus I_4\oplus I_5\oplus I_6$.
Note that $P_3\cong I_2, P_5\cong I_2, P_6\cong P_3$.
Then the $2$-Auslander algebra $\Lambda^{(3,3)}$ and the total $3$-preprojective algebra $\Psi^{(3,3)}$ of $\Lambda=\Lambda^{(2,3)}$ can be presented as follows.
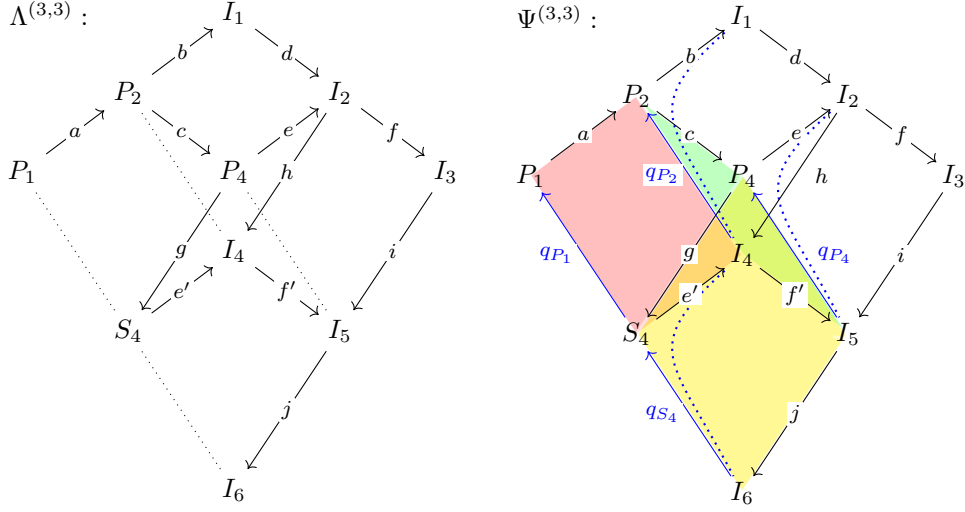
\begin{center}
\begin{tikzpicture}[scale=0.7]
\tikzset{
  sn/.style={fill=white,font=\footnotesize,inner sep=1pt},
}
\node (v1) at (0,0) {$P_1$};
\node (v2) at (2,1.5) {$P_2$};
\node (v3) at (4,3) {$I_1$};
\node (v4) at (4,0) {$P_4$};
\node (v5) at (6,1.5) {$I_2$};
\node (v6) at (8,0) {$I_3$};
\node (v9) at (2,-3) {$S_4$};
\node (v7) at (4,-1.5) {$I_4$};
\node (v8) at (6,-3) {$I_5$};
\node (v10) at (4,-6) {$I_6$};
\draw[dotted]  (v9) -- (v1);
\draw[dotted]  (v8) -- (v4);
\draw[dotted]  (v7) -- (v2);
\draw[dotted]  (v10) -- (v9);
\tikzset{
  every path/.append style = {->, inner sep=2pt},
}
\draw  (v1) to node[sn] {$a$} (v2);
\draw  (v2) to node[sn] {$b$} (v3);
\draw  (v2) to node[sn] {$c$} (v4);
\draw  (v3) to node[sn] {$d$} (v5);
\draw  (v4) to node[sn] {$e$} (v5);
\draw  (v5) to node[sn] {$f$} (v6);
\draw  (v7) to node[sn] {$f'$} (v8);
\draw  (v9) to node[sn] {$e'$} (v7);
\draw  (v5) to node[sn] {$h$} (v7);
\draw  (v6) to node[sn] {$i$} (v8);
\draw  (v8) to node[sn] {$j$} (v10);
\draw  (v4) to node[sn] (a6) {$g$} (v9);
\node at (0.5,3) {$\Lambda^{(3,3)}:$};
\end{tikzpicture}\ \ \ \ \ 
\begin{tikzpicture}[scale=0.7]
\tikzset{
  sn/.style={fill=white,font=\footnotesize,inner sep=1pt},
  every path/.append style = {->, inner sep=2pt},
}
\filldraw[draw=white,fill=red!50!white,opacity=.5,inner sep=2pt] (0,0) -- (2,1.5) -- (4,-1.5) -- (2,-3) -- cycle;
\filldraw[draw=white,very thick,fill=green!50!white,opacity=.5] (4,0) -- (2,1.5) -- (4,-1.5) -- (6,-3) ;
\filldraw[draw=white,very thick,fill=yellow!90!white,opacity=.5] (4,0) -- (2,-3) -- (4,-6) -- (6,-3) ;
\node (v1) at (0,0) {$P_1$};
\node (v2) at (2,1.5) {$P_2$};
\node (v3) at (4,3) {$I_1$};
\node (v4) at (4,0) {$P_4$};
\node (v5) at (6,1.5) {$I_2$};
\node (v6) at (8,0) {$I_3$};
\node (v9) at (2,-3) {$S_4$};
\node (v7) at (4,-1.5) {$I_4$};
\node (v8) at (6,-3) {$I_5$};
\node (v10) at (4,-6) {$I_6$};
\draw[-,dotted,thick,blue] (v3) .. controls +(-2,-1.8) and +(-1.3,2.5) .. (v7);
\draw[-,dotted,thick,blue] (v7) .. controls +(-2,-1.8) and +(-1.3,2.5) .. (v10);
\draw[-,dotted,thick,blue] (v5) .. controls +(-2,-1.8) and +(-1.3,2.5) .. (v8);
\draw  (v1) to node[sn] {$a$} (v2);
\draw  (v2) to node[sn] {$b$} (v3);
\draw  (v2) to node[sn] {$c$} (v4);
\draw  (v3) to node[sn] {$d$} (v5);
\draw  (v4) to node[sn] {$e$} (v5);
\draw  (v5) to node[sn] {$f$} (v6);
\draw  (v7) to node[sn] {$f'$} (v8);
\draw  (v9) to node[sn] {$e'$} (v7);
\draw  (v5) to node[sn,xshift={10}] {$h$} (v7);
\draw  (v6) to node[sn] {$i$} (v8);
\draw  (v8) to node[sn] {$j$} (v10);
\draw  (v4) to node[sn] {$g$} (v9);
\draw[blue]  (v9) to node[sn,xshift={-10}] {$q_{P_1}$} (v1);
\draw[blue]  (v8) to node[sn,xshift={15}] {$q_{P_4}$} (v4);
\draw[blue]  (v7) to node[sn,xshift={-10}] {$q_{P_2}$} (v2);
\draw[blue]  (v10) to node[sn,xshift={-10}] {$q_{S_4}$}  (v9);
\node at (0.5,3) {$\Psi^{(3,3)}:$};
\end{tikzpicture}
\end{center}
The relations of $\Lambda^{(3,3)}$ are given by commutation of the rectangular faces and the quadratic monomial at the boundary (such as $ca$ and $gc$).
For $\Psi^{(3,3)}$, the relations are given by the blue dotted lines, which represent the quadratic monomials $r_b=bq_{P_2}, r_e=eq_{P_4}, r_{e'}=e'q_{S_4}$,   
and the coloured faces, which represent the three commutations relations $r_a=aq_{P_1}-q_{P_2}e', r_c=cq_{P_2}-q_{P_4}f', r_g=gq_{S_4}-q_{P_4}j$.

\subsection{The quadratic dual of the stable Auslander algebra of type $D_4$}

The last example comes from the quadratic dual of stable Auslander algebras of Dynkin quivers \cite{M}. For details we refer to \cite{CIM2}.
Let $\Lambda$ be the algebra given by the quiver
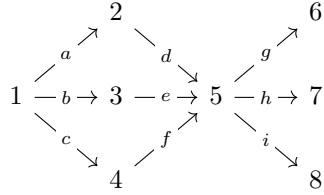
\[\begin{tikzcd}
	&2 && 6 \\
	1 & 3 & 5 & 7 \\
	&4 && 8
	\arrow["d"{description}, from=1-2, to=2-3]
	\arrow["g"{description}, from=2-3, to=1-4]
	\arrow["h"{description}, from=2-3, to=2-4]
	\arrow["i"{description}, from=2-3, to=3-4]
	\arrow["e"{description}, from=2-2, to=2-3]
	\arrow["a"{description}, from=2-1, to=1-2]
	\arrow["b"{description}, from=2-1, to=2-2]
	\arrow["c"{description}, from=2-1, to=3-2]
	\arrow["f"{description}, from=3-2, to=2-3]
\end{tikzcd}\]
with relations $da-eb, da-fc, gd, hd, ge, ie, hf, if$.
Then $\Lambda$ is $3$-representation-finite with $3$-cluster tilting module $M=D(\Lambda) \oplus \tau_3(D(\Lambda))$.
The $3$-Auslander algebra $\Gamma=\End_\Lambda(M)$ has the quiver
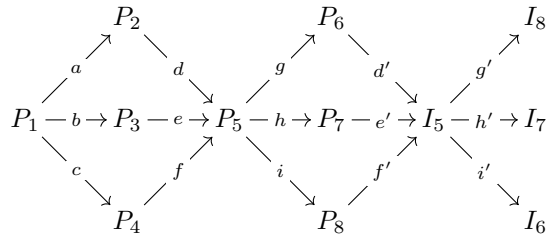
\[\begin{tikzpicture}[scale=0.9]
\tikzset{
 every node/.append style = {fill=white},
  every path/.append style = {->,inner sep=2pt},
}
\node (v1) at (0,0) {$P_1$};
\node (v2) at (1.5,1.5) {$P_2$};
\node (v3) at (1.5,0) {$P_3$};
\node (v4) at (1.5,-1.5) {$P_4$};
\node (v5) at (3,0) {$P_5$};
\node (v6) at (4.5,1.5) {$P_6$};
\node (v7) at (4.5,0) {$P_7$};
\node (v8) at (4.5,-1.5) {$P_8$};
\node (v9) at (6,0) {$I_5$};
\node (v10) at (7.5,1.5) {$I_8$};
\node (v11) at (7.5,0) {$I_7$};
\node (v12) at (7.5,-1.5) {$I_6$};
\draw  (v1) to node {{\scriptsize$a$}} (v2) ;
\draw  (v1) to node {{\scriptsize$b$}} (v3);
\draw  (v1) to node {{\scriptsize$c$}} (v4);
\draw  (v2) to node {{\scriptsize$d$}} (v5);
\draw  (v3) to node {{\scriptsize$e$}} (v5);
\draw  (v4) to node {{\scriptsize$f$}} (v5);
\draw  (v5) to node {{\scriptsize$g$}} (v6);
\draw  (v5) to node {{\scriptsize$h$}} (v7);
\draw  (v5) to node {{\scriptsize$i$}} (v8);
\draw  (v6) to node {{\scriptsize$d'$}} (v9);
\draw  (v7) to node {{\scriptsize$e'$}} (v9);
\draw  (v8) to node {{\scriptsize$f'$}} (v9);
\draw  (v9) to node {{\scriptsize$g'$}} (v10);
\draw  (v9) to node {{\scriptsize$h'$}} (v11);
\draw  (v9) to node {{\scriptsize$i'$}} (v12);
\end{tikzpicture}\]
with relations $da-eb, da-fc, gd, hd, ge, ie, hf, if, d'g-e'h, d'g-f'i, g'd', h'd', g'e', i'e', h'f', i'f'$.
The total $4$-preprojective algebra $\Psi$ has the quiver
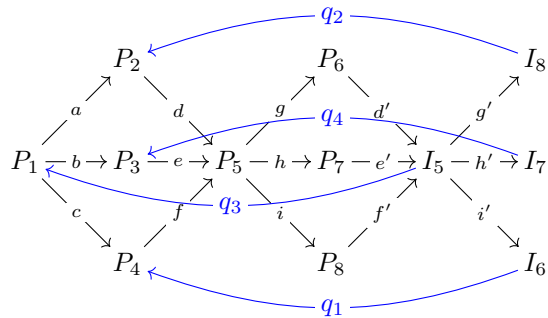
\[\begin{tikzpicture}[scale=0.9]
\tikzset{
 every node/.append style = {fill=white},
  every path/.append style = {->,inner sep=2pt},
}
\node (v1) at (0,0) {$P_1$};
\node (v2) at (1.5,1.5) {$P_2$};
\node (v3) at (1.5,0) {$P_3$};
\node (v4) at (1.5,-1.5) {$P_4$};
\node (v5) at (3,0) {$P_5$};
\node (v6) at (4.5,1.5) {$P_6$};
\node (v7) at (4.5,0) {$P_7$};
\node (v8) at (4.5,-1.5) {$P_8$};
\node (v9) at (6,0) {$I_5$};
\node (v10) at (7.5,1.5) {$I_8$};
\node (v11) at (7.5,0) {$I_7$};
\node (v12) at (7.5,-1.5) {$I_6$};
\draw  (v1) to node {{\scriptsize$a$}} (v2) ;
\draw  (v1) to node {{\scriptsize$b$}} (v3);
\draw  (v1) to node {{\scriptsize$c$}} (v4);
\draw  (v2) to node {{\scriptsize$d$}} (v5);
\draw  (v3) to node {{\scriptsize$e$}} (v5);
\draw  (v4) to node {{\scriptsize$f$}} (v5);
\draw  (v5) to node {{\scriptsize$g$}} (v6);
\draw  (v5) to node {{\scriptsize$h$}} (v7);
\draw  (v5) to node {{\scriptsize$i$}} (v8);
\draw  (v6) to node {{\scriptsize$d'$}} (v9);
\draw  (v7) to node {{\scriptsize$e'$}} (v9);
\draw  (v8) to node {{\scriptsize$f'$}} (v9);
\draw  (v9) to node {{\scriptsize$g'$}} (v10);
\draw  (v9) to node {{\scriptsize$h'$}} (v11);
\draw  (v9) to node {{\scriptsize$i'$}} (v12);
\draw[blue]  (v10) to[bend right=20] node {$q_2$} (v2);
\draw[blue]  (v9) to[bend left=20] node {$q_3$} (v1);
\draw[blue]  (v12) to[bend left=20] node {$q_1$} (v4);
\draw[blue]  (v11) to[bend right=20] node {$q_4$} (v3);
\end{tikzpicture}\]
with additional relations $r_a = aq_1-q_2g', r_b = bq_3-q_5e', r_c=cq_1-q_4i', r_d=dq_2, r_e=eq_3, r_f=fq_4.$


\end{document}